\documentclass[reqno]{amsart}

\usepackage{amsmath,amssymb,amscd,amsthm,accents} \usepackage{graphicx}
\DeclareGraphicsExtensions{.eps} \usepackage{mathrsfs}
\usepackage[mathcal]{eucal}
\usepackage{enumitem}
\usepackage{mathtools}

\def\sideremark#1{\ifvmode\leavevmode\fi\vadjust{\vbox to0pt{\vss 
			\hbox to 0pt{\hskip\hsize\hskip1em          
				\vbox{\hsize3cm\tiny\raggedright\pretolerance10000%
					\noindent #1\hfill}\hss}\vbox to8pt{\vfil}\vss}}}%

\newtheorem{Thm}{Theorem}{\bfseries}{\itshape}
\newtheorem*{Thm*}{Theorem}{\bfseries}{\itshape}
\newtheorem{Cor}{Corollary}{\bfseries}{\itshape}
\newtheorem{Prop}[Cor]{Proposition}{\bfseries}{\itshape}
\newtheorem{Lem}[Cor]{Lemma}{\bfseries}{\itshape}
\newtheorem*{Lem*}{Lemma}{\bfseries}{\itshape}
{\bfseries}{\itshape}
{\bfseries}{\itshape}
\newtheorem{Def}[Cor]{Definition}{\bfseries}{\rmfamily}
{\scshape}{\rmfamily}
\newtheorem{Rem}[Cor]{Remark}{\scshape}{\rmfamily}
{\bfseries}{\itshape}

{\scshape}{\rmfamily}

\renewcommand\ge{\geqslant} \renewcommand\le{\leqslant}
\let\tildeaccent=\~ \let\hataccent=\^
\renewcommand\~[1]{\widetilde{#1}}

\def\<{\left<} \def\>{\right>} \def\({\left(} \def\){\right)}

\def\norm#1{\Vert #1\Vert}

\let\parasymbol=\S \def\secref#1{\parasymbol\ref{#1}}

\def\pd#1#2{\frac{\partial#1}{\partial#2}}

\let\polishL=l \def\Zoladek.{\.Zol\c adek}

 \def\Im{\operatorname{Im}}
 \def\dist{\operatorname{dist}}

\def\etc.{\emph{etc}.}

 \def\R{{\mathbb R}}  \def\Z{{\mathbb
    Z}} \def\N{{\mathbb N}} \def\Q{{\mathbb Q}}

 \def\e{\varepsilon} \def\S{\varSigma}

\def\poly{\operatorname{poly}}

 \def\d{\,\mathrm d}

\def\cF{{\mathcal F}}

\def\cS{{\mathcal S}}

\def\cM{{\mathcal M}}

 \def\trans{\pitchfork}

\def\rest#1{{\vert_{#1}}}

\def\vx{{\mathbf x}}

\def\valpha{{\boldsymbol\alpha}}

\def\alg{\mathrm{alg}}
\def\trans{\mathrm{trans}}

\DeclareMathOperator{\gr}{gr}

\def\rPfaff{{\mathrm{rPfaff}}}  
\newcommand{\sube}[1][\e]{{\subseteqq_{#1}}}
\def\so{\raisebox{.5ex}{\scalebox{0.6}{\#}}\kern-.02em{}o}
\def\an{{\mathrm{an}}}
\def\PF{{P_{\cF}}}

\begin{document}

\title{Wilkie's conjecture for Pfaffian structures}

\author{Gal Binyamini}
\address{Weizmann Institute of Science, Rehovot, Israel}
\email{gal.binyamini@weizmann.ac.il}

\author{Dmitry Novikov}
\address{Weizmann Institute of Science, Rehovot, Israel}
\email{dmitry.novikov@weizmann.ac.il}

\author{Benny Zack}
\address{Weizmann Institute of Science, Rehovot, Israel}
\email{binyamin.zackkutuzov@weizmann.ac.il}

\thanks{This research was supported by the ISRAEL SCIENCE FOUNDATION
  (grant No. 1167/17) and by funding received from the MINERVA
  Stiftung with the funds from the BMBF of the Federal Republic of
  Germany. This project has received funding from the European
  Research Council (ERC) under the European Union's Horizon 2020
  research and innovation programme (grant agreement No 802107)}


\begin{abstract}
  We prove an effective form of Wilkie's conjecture in the structure
  generated by restricted sub-Pfaffian functions: the number of
  rational points of height $H$ lying in the transcendental part of
  such a set grows no faster than some power of $\log H$. Our bounds
  depend only on the Pfaffian complexity of the sets involved. As a
  corollary we deduce Wilkie's original conjecture for $\R_{\exp}$ in
  full generality.
\end{abstract}
\maketitle

\section{Introduction}

\subsection{Main results}

Let $(\cS,\Omega)$ be a \so-minimal structure admitting sharp cell
decomposition and sharp derivatives (for the definition of these
notions see~\secref{sec:so-minimal}). The structure $\R_\rPfaff$ of
restricted sub-Pfaffian sets (see~\secref{sec:rPfaff}) is an example
of this setup, and the unfamiliar reader may keep this example in mind
in place of the general setting.

If $X\subset\R^n$ then following \cite{pw:thm} we denote by $X^\alg$
the union of all connected, positive-dimensional semialgebraic sets
contains in $X$, and denote $X^\trans:=X\setminus X^\alg$. For
$g,H\in\N$ we denote
\begin{equation}
  X(g,H) := \{x\in X : [\Q(x):\Q]\le g,\ H(x)\le H \}, \qquad X(\Q,H):=X(1,H)
\end{equation}
where $H(\cdot)$ denotes the multiplicative Weil height on $\bar\Q$,
extended to $\bar\Q^n$ as the maximum of the heights of the
coordinates. If unfamiliar see~\secref{sec:asymptotics} for the
asymptotic notation used below.

\begin{Thm}\label{thm:main}
  Let $X\in\Omega_{\cF,D}$. Then
  \begin{equation}
    \# X^\trans(g,H)\le\poly_\cF(D,g,\log H).
  \end{equation}
\end{Thm}

This establishes, in the restricted Pfaffian setting, a conjecture by
Pila \cite[Conjecture~1.5]{pila:pfaff}. As an immediate corollary we
obtain the following.

\begin{Cor}[Wilkie's conjecture]\label{cor:wilkie}
  Let $X$ be definable in $\R_{\exp}$. Then
  \begin{equation}
    \# X^\trans(g,H)\le\poly_X(g,\log H).
  \end{equation}
\end{Cor}
\begin{proof}
  By Wilkie's theorem of the complement \cite{wilkie:complements} we
  have $X=\pi_n(Y)$ where $Y\subset\R^N$ is quantifier-free in
  $\R_{\exp}$, and $\pi_n:\R^N\to\R^n$ is the projection to the first
  $n$ coordinates. Let $g,H\in\N$ and choose $M\gg1$ such that
  $X(g,H)=X_M(g,H)$ where
  \begin{equation}
    X_M := \pi_n(Y_M), \quad Y_M:=Y\cap[-M,M]^N.
  \end{equation}
  Now $Y_M$ is restricted semi-Pfaffian, as it is defined by Pfaffian
  functions (exponential polynomials) restricted to
  $[-M,M]^N$. Crucially, $Y_M\in\Omega_{\cF,D}$ where $\cF,D$ depend on
  $Y$ but not on $M$.  Then the same is true for $X_M$, and we
  conclude
  \begin{equation}
    \#X^\trans(g,H)\le\#X_M^\trans(g,H)=\poly_X(g,\log H)
  \end{equation}
  by Theorem~\ref{thm:main}.
\end{proof}

We also have a ``blocks'' generalization of
Theorem~\ref{thm:main}. Recall from \cite{pila:blocks} that a
definable set $B\subset\R^n$ is called a \emph{basic block} if it is
connected and regular, and contained in a connected regular
semialgebraic set of the same dimension (which we call a semialgebraic
closure of $B$, though not this is not uniquely defined). We denote by
$\Omega^\alg$ a filtration making $(\R_\alg,\Omega^\alg)$ into a
\so-minimal structure (one can take, e.g., the filtration from
\cite{bv:rpfaff} for the empty Pfaffian chain).

\begin{Thm}\label{thm:main-blocks}
  Let $X\subset\R^n$ with $X\in\Omega_{\cF,D}$. Then there exists a
  collection $\{B_\eta\subset X\}$ of basic blocks with semialgebraic
  closures $S_\eta$ such that $X(g,H) \subset \cup_\eta B_\eta$ and
  \begin{equation}
    \#\{B_\eta\}=\poly_\cF(D,g,\log H), \qquad \forall\eta: S_\eta\in\Omega^\alg_{O_n(1),\poly_n(g,\log H)}.
  \end{equation}
\end{Thm}

Theorem~\ref{thm:main-blocks} clearly implies Theorem~\ref{thm:main},
since the positive-dimensional basic blocks $B_\eta$ are subsets of
$X^\alg$ by definition.

\subsection{A $C^r$-parametrization lemma}
\label{sec:intro-Cr}

For a $C^r$-smooth function $f:U\to\R$ on a domain $U\subset\R^m$ we
denote
\begin{align}
  \norm{f} &:= \sup_{x\in U} |f(x)|, &  \norm{f}_r &:= \max_{|\valpha|\le r} \norm{D^\valpha f}.
\end{align}
For $F:U\to\R^n$ we set $\norm{F}=\max_{i}\norm{F_i}$ and similarly
for $\norm{F}_r$. For a set $A\subset\R^n$ we write $U_\e(A)$ for the
$\e$-neighborhood of $A$ with the $\ell_\infty$-metric.  For
$A,B\subset\R^n$, we write $A\sube B$ to mean that $A\subset B$ and
$B\subset U_\e(A)$. We say that $A$ is an \emph{$\e$-cover} of $B$.

The main novelty of our approach is the following version of Yomdin's
algebraic lemma. Let $I:=(0,1)$.
\begin{Lem}\label{lem:alg-lem}
  Let $r\in\N$ and $\e>0$.  Let $X\subset[0,1]^n$ be of dimension
  $\mu$ with $X\in\Omega_{\cF,D}$. Then there exists a collection
  $\{\phi_\eta:I^\mu\to X\}$ such that
  $\norm{\phi_\eta}_r\le 1$ and $\cup_\eta \Im\phi_\eta\sube X$,
  and
  \begin{equation}
    \#\{\phi_\eta\} \le \poly_\cF(D,r,|\log\e|), \qquad \forall\eta: \phi_\eta\in\Omega_{O_\cF(1),\poly_\cF(D,r)}.
  \end{equation}
\end{Lem}

This formulation is similar in spirit to Yomdin's original formulation
\cite{yomdin:lemma,yomdin:lemma-addendum}. Gromov \cite{gromov:gy}
later refined Yomdin's work by showing that we can avoid $\e$-covers
altogether and cover the set $X$ completely. However, as we will see,
Lemma~\ref{lem:alg-lem} is sufficient for the applications in
Diophantine geometry (as it was for Yomdin's original application in
dynamics). The weaker formulation with $\e$-covers enables us (as was
already the case in Yomdin's original work) to restrict to
\emph{affine} reparametrizations at some crucial moments, where
Gromov's approach involves nonlinear terms. The linearity of the
reparametrizing maps turns out to allow for crucial technical
simplifications related to achieving polynomial growth of
$\#\{\phi_\eta\}$ as a function of $r$ (specifically
in~\secref{sec:higher-der})

\subsection{Background}

\subsubsection{The Pila-Wilkie theorem}
The origin of the area of point-counting in tame geometry can be
traced to the work of Bombieri and Pila \cite{bp,pila:Q-count}. In
these papers it was shown that if $\Gamma\subset\R^2$ is a compact
analytic curve containing no semialgebraic curves then for every
$\e>0$ one has $\#\Gamma(\Q,H)=O_{\Gamma,\e} (H^\e)$. After some work
by Pila on subanalytic surfaces
\cite{pila:subanalytic-integer,pila:subanalytic-rational}, this result
was generalized into its canonical form by Pila and Wilkie
\cite{pw:thm}, who proved that the bound
$\#X^\trans(\Q,H)<C(X,\e)\cdot H^\e$ holds for \emph{any} $X$
definable in an o-minimal structure. This result has had a profound
impact on arithmetic geometry, and we refer the reader to
\cite{scnalon:survey} for a survey.

For Pfaffian surfaces, Jones and Thomas \cite{jt:effective-pw-surface}
established an effective form of the Pila-Wilkie theorem. In the
general restricted sub-Pfaffian setting, a recent paper by the first
author with Jones, Schmidt and Thomas \cite{bjst:effective-pw}
establishes an effective form of the Pila-Wilkie theorem: if
$X\in\Omega_{\cF,D}$ then one can take
$C(X,\e)=\poly_{\cF,\e}(D)$. Many of the technical methods for
using \so-minimality in our context are inspired by this prior
work. Indeed the paper \cite{bv:rpfaff} which inspired the notion of
\so-minimality grew out of an attempt to provide a suitable foundation
for the results in \cite{bjst:effective-pw}.

\subsubsection{The Wilkie conjecture}
Examples by Pila \cite[Example~7.5]{pila:subanalytic-integer} show
that in $\R_\an$ the Pila-Wilkie asymptotic is essentially
optimal. However, such examples involve ``hand-crafted'' functions and
no ``natural'' example exhibiting this behavior is known. Wilkie made his
conjecture (now Corollary~\ref{cor:wilkie}) in the original paper
\cite{pw:thm} as a concrete formulation of this phenomenon. The case
of Pfaffian curves was proved by Pila \cite{pila:pfaff}, and our
approach in the one-dimensional case is indeed somewhat similar to
Pila's approach. Some further examples of surfaces were treated in
\cite{pila:exp-surface}, but general surfaces already seem difficult
to treat with this approach.

The key obstacle to progress on Wilkie's conjecture has been to
establish a $C^r$-parametrization lemma with polynomial bounds for the
number of charts, as a function of the complexity of the set and the
smoothness order $r$. This problem was open even in the semialgebraic
case, and was recently resolved in \cite{me:complex-cells} using
complex analytic methods (see also
\cite{cpw:params,vhille:mild,vhille:power-subanalytic} for a result on
polynomial growth in $r$, without complexity bounds in some o-minimal
structures). The problem remains open beyond the semialgebraic case,
and the complex analytic methods seem unlikely to directly carry over
to the unrestricted exponential case. For general definable sets, the
only previously known case of the Wilkie conjecture is
\cite{me:rest-wilkie} by the first two authors. This paper established
Wilkie's conjecture for the structure $\R^{\mathrm{RE}}$ generated by
the exponential and sine functions \emph{restricted} to compact
domains. The proofs were based on an approach avoiding smooth
parametrizations altogether, replacing it by complex-geometric
ideas. This is only applicable for \emph{holomorphic-Pfaffian}
functions, i.e. holomorphic functions whose graph, viewed as a real
set, is Pfaffian in the real sense. By comparison, our approach here
applies to arbitrary restricted sub-Pfaffian functions without
requiring that the complex-analytic continuation is again Pfaffian.
The complex-geometric ideas also seem much more difficult to carry out
in the presence of unrestricted exponentials.

\begin{Rem}
  The proof of Corollary~\ref{cor:wilkie} would not be applicable with
  the methods of \cite{me:rest-wilkie} because the complex analytic
  nature of these methods would require us to consider
  $e^{z_1},\ldots,e^{z_M}$ restricted to large \emph{complex
    polydiscs} $D(0,M)^N$ rather than large cubes $[-M,M]^N$ as we do
  here. However, while the Pfaffian complexity of $e^x$ on $[-M,M]$ is
  bounded independently of $M$, the Pfaffian complexity of $e^z$ on
  $D(0,M)$ is roughly $M$ as evidenced by the fact that the Pfaffian
  equation $e^z=1$ admits roughly $M/\pi$ solutions in $D(0,M)$.
\end{Rem}

\subsubsection{Unrestricted exponentials in arithmetic applications}

Unrestricted exponentials are used in many of the most spectacular
applications of the Pila-Wilkie theorem, where they arise in
uniformizing maps of arithmetic quotients around cusps. Extending the
more advanced counting techniques to this case is therefore
potentially very useful. In particular, a recent paper by the first
author \cite{me:qfol-geometry} establishes a polylogarithmic counting
result in the spirit of Wilkie's conjecture for sets defined using
algebraic foliations (not necessarily Pfaffian) over number
fields. This result has played an important role in recent progress on
the Andr\'e-Oort conjecture for general Shimura varieties.  It was
used by the first author, Schmidt and Yafaev \cite{bsy:ao} to
establish Galois lower bounds for special points conditional on
certain height bounds. These height bounds were subsequently proved by
Pila, Shankar and Tsimerman thus finishing the proof of Andr\'e-Oort
in general.

The approach of \cite{me:qfol-geometry} is based on the complex
geometric ideas of \cite{me:rest-wilkie} and suffers from the same
limitation to restricted analytic situations, and this leads to
technical complications in \cite{bsy:ao} and in further potential
applications of this result in arithmetic geometry. It seems plausible
that the new approach developed in the present paper could also lead
to progress on unrestricted exponentials in this non-Pfaffian
situation, and we have formulated our results in the more general
\so-minimal context with this in mind.

\subsection{Asymptotic notation}
\label{sec:asymptotics}

In this paper each appearance of an expression $Z=O_X(Y)$ should be
interpreted as shorthand notation for $Z\le C_X\cdot Y$ where
$X\to C_X$ is a universally fixed, positive valued real
function. Similarly we write $Z=\poly_X(Y)$ as shorthand for
$Z\le P_X(Y)$ where $X\to P_X$ is a universally fixed mapping and
$P_X$ is a polynomial with positive coefficients. However we suppress
dependence of the constants on $(\cS,\Omega)$, which we consider to be
universally fixed throughout the text.

\subsection{Acknowledgments}

It is our pleasure to thank Yosi Yomdin for insightful discussions on
the algebraic lemma, and Alex Wilkie for alerting us of the potential
relevance of his notes \cite{wilkie:notes}. In these notes Wilkie
introduces a method for obtaining $C^r$-parametrizations in the
one-dimensional case with a single reparametrization, rather than by
the more traditional induction on $r$. While we did not eventually use
this directly in our text, our approach is a kind of discrete version
of this idea (so that we can use linear reparametrizations similar to
Yomdin's approach) and certainly inspired by it. Pila
\cite{pila:pfaff} has used a similar approach earlier for Pfaffian
curves, and his idea also inspires our approach. We note further that
the interpolation method that we use to control $X(g,H)$ efficiently
as a function of $g$ was also introduced in Wilkie's important notes
\cite{wilkie:notes}.

\section{Sharply o-minimal structures}
\label{sec:so-minimal}

\subsection{\so-minimal structures}
\label{sec:so-minimality}

In this section we introduce the notion of a \emph{sharply o-minimal
  structure} (abbreviated \emph{\so-minimal}). To start, a
\emph{format-degree} filtration (abbreviated \emph{FD-}filtration) on
a structure $\cS$ is a collection
$\Omega=\{\Omega_{\cF,D}\}_{\cF,D\in\N}$ such that each
$\Omega_{\cF,D}$ is a collection of definable sets (possibly of
different ambient dimensions), with
$\Omega_{\cF,D} \subset \Omega_{\cF+1,D}\cap \Omega_{\cF,D+1}$ and
$\cup_{\cF,D}\Omega_{\cF,D}$ is the collection of all definable sets
in $\cS$. We call the sets in $\Omega_{\cF,D}$ sets of \emph{format}
$\cF$ and \emph{degree} D. We will assume $\Omega_{\cF,D}$ only
contains subsets of $\R^n$ for $n\le\cF$.

A \so-minimal structure is a pair $\Sigma:=(\cS,\Omega)$
consisting of an o-minimal expansion of the real field $\cS$ and an
FD-filtration $\Omega$; and for each $\cF\in\N$ a
polynomial $P_\cF(\cdot)$ such that the following holds:
\begin{enumerate}
\item If $A\in \Omega_{\cF,D}$ with $A\subset\R^n$ then
  $A^c,\pi_{n-1}(A),A\times\R$ and $\R\times A$ lie in
  $\Omega_{\cF+1,D}$.
\item If $A_1,\ldots,A_k\subset\R^n$ with $A_j\in \Omega_{\cF,D_j}$
  then $\cup_{i}A_i\in \Omega_{\cF,D}$ and
  $\cap_{i}A_i\in \Omega_{\cF+1,D}$ where $D=\sum_j D_j$.
\item\label{axiom:hypersurface} If $P\in\R[x_1,\ldots,x_n]$ then $\{P=0\}\in \Omega_{n,\deg P}$.
\item If $A\in\Omega_{\cF,D}$ with $A\subset\R$ then it has at most
  $\PF(D)$ connected components,
\end{enumerate}
Axioms 1-2 bear a close analogy to the standard axioms of a
first-order structure, keeping track of the formats and degrees of
sets defined using the logical operations. Axiom 3 ensures
compatibility with the standard notion of degree in the
\mbox{(semi-)}algebraic case. Finally Axiom 4 replaces the mere finiteness
postulated in standard o-minimality by polynomial bounds in degrees.

\subsection{Sharp cell decomposition}

The following notion is crucial for working with \so-minimal
structures.

\begin{Def}
  We say that $(\cS,\Omega)$ has \emph{sharp cell decomposition} if
  for every $\cF\in\N$ there are
  \begin{align}
    a_\cF&\in\N, &  b_\cF&\in\N[D,k], & c_\cF&\in\N[D]
  \end{align}
  such that the following holds. For every $\cF,D\in\N$ and every
  $X_1,\ldots,X_k\in\Omega_{\cF,D}$ subsets of $\R^n$, there exists a
  cylindrical decomposition $\{C_\eta\}$ of $\R^n$ compatible with
  $X_1,\ldots,X_k$ such that
  \begin{equation}
    \#\{C_\eta\} \le b_\cF(D,k), \qquad \forall\eta: C_\eta\in\Omega_{a_\cF,c_\cF(D)}.
  \end{equation}
\end{Def}

We use the following notation for cells from
\cite{me:complex-cells}. For $C\subset\R^{n-1}$ and $a,b:C\to\R$ we
set
\begin{align}
  \begin{aligned}
    C\odot\{a(z)\} &:= \{ (z,w) : z\in C,\ w=a(z)\}, \\
    C\odot(a(z),b(z)) &:= \{ (z,w) : z\in C,\ a(z)<w<b(z) \}.
  \end{aligned}
\end{align}
We will also allow $a(z)\equiv-\infty$ and $b(z)\equiv\infty$ in the
second case above.

In an upcoming paper we prove, based on ideas from \cite{bv:rpfaff},
that for every \so-minimal structure $(\cS,\Omega)$ there is another
FD-filtration $\Omega^*$ with $\Omega_{\cF,D}\subset\Omega^*_{\cF,D}$
for every $\cF,D$ such that $(\cS,\Omega^*)$ is \so-minimal with sharp
cell decomposition. This implies that Theorem~\ref{thm:main-blocks}
and its consequences actually apply without explicitly assuming that
$(\cS,\Omega)$ has sharp cell decomposition. However to keep matters
clear and avoid dependence on our upcoming text we keep this as an
extra condition. Our main example $\R_\rPfaff$ does, in any case,
admit sharp cell decomposition as explained in~\secref{sec:rPfaff}.

\subsection{Some consequences of \so-minimality and sharp cell decomposition}

The axioms of \so-minimality imply that whenever
$X_1,\ldots,X_k\in\Omega_{\cF,D}$ and $\psi$ is a first-order formula
of depth $\ell$ with basic predicates $\vx\in X_j$ then the set $X$
defined by $\phi$ satisfies
\begin{equation}
  X\in\Omega_{O_{\cF,\ell}(1),\poly_{\cF,\ell}(D,k)},
\end{equation}
see \cite[Section~1.3]{bv:rpfaff} for a more precise treatment. Together
with sharp cell decomposition, this can be used to effectivize many of
the classical constructions of o-minimality in a rather routine
fashion. We record a few instances used in our text to familiarize the
reader with this technique.

\begin{Prop}[Connected components]
  Let $X\in\Omega_{\cF,D}$. Then each connected component of $X$ is in
  $\Omega_{O_\cF(1),\poly_\cF(D)}$, and their number is
  $\poly_{\cF}(D)$.
\end{Prop}
\begin{proof}
  Perform a cell decomposition. Each connected component is a union of
  cells.
\end{proof}

\begin{Prop}[Stratification]
  Let $X\in\Omega_{\cF,D}$. Then $X$ is a disjoint union
  $\cup_\eta S_\eta$ where each $S_\eta$ is connected and regular, and
  \begin{equation}
    \#\{S_\eta\} = \poly_\cF(D), \qquad \forall\eta: S_\eta\in\Omega_{O_\cF(1),\poly_\cF(D)}.
  \end{equation}
\end{Prop}
\begin{proof}
  Let $\mu:=\dim X$. Let $S\subset X$ be the $\mu$-regular part of
  $X$, i.e. the set of points $p\in X$ such that for some linear
  projection $L:\R^n\to\R^\mu$, the map $L\rest{X}$ is locally
  invertible at $p$, and the inverse
  $L':(\R^\mu,L(p))\to X\subset\R^n$ is locally $C^1$ with Jacobian of
  rank $\mu$. This can be written out explicitly as a first-order
  formula in $\e$-$\delta$-type language, so the axioms of
  \so-minimality give $S\in\Omega_{O_\cF(1),\poly_\cF(D)}$. Each
  connected component of $S$ is a top-dimensional stratum, and the
  remaining set $X\setminus S$ can be handled by induction on $\mu$.
\end{proof}

\begin{Prop}[Definable choice]
  Let $X\subset\Lambda\times\R^n$ with $X\in\Omega_{\cF,D}$, and
  suppose $X_\lambda\neq\emptyset$ for every $\lambda\in\Lambda$. Then
  there is a map $F:\Lambda\to\R^n$ with
  $\gr F\in\Omega_{O_\cF(1),\poly_\cF(D)}$ such that $\gr F\subset X$.
\end{Prop}
\begin{proof}
  Perform a cylindrical decomposition of $\Lambda\times\R^n$
  compatible with $X$. In particular we obtain a cylindrical
  decomposition $\{C_\eta\}$ of $\Lambda$, and over each $C_\eta$ a
  cylindrical decomposition of $C_\eta\times\R^n$ by cells projecting
  to $C_\eta$. It will be enough to handle each $C_\eta$ separately
  and then take the unions of the corresponding graphs. Moreover, we
  may as well consider just one of the cells over $C_\eta$ that is
  contained in $X$ for the purpose of defining the choice function. So
  assume without loss of generality that $X$ is a cell.

  Write $X=C\odot(a(z),b(z))$ where
  $C\in\Omega_{O_\cF(1),\poly_\cF(D)}$ is a cell in
  $\Lambda\times\R^{n-1}$ and $a(z),b(z):C\to\R$. The cases
  $a(z)=-\infty$, $b(z)=\infty$ and $C\odot\{a(z)\}$ are treated
  similarly. We have
  $\gr a(z),\gr b(z)\in\Omega_{O_\cF(1),\poly_\cF(D)}$ since they can
  be defined using first-order formulas as the infimum and supremum of
  the fiber $C_z$. Then we find a choice function $\hat F(\lambda)$ on
  $C$ by induction on $n$, and
  \begin{equation}
    F(\lambda):=\bigg(\hat F(\lambda),\frac{a(\lambda,\hat F(\lambda))+b(\lambda,\hat F(\lambda))}{2}\bigg)
  \end{equation}
  is a choice function for $X$.
\end{proof}

\subsection{Sharp derivatives}

If $f:X\to Y$ is a definable function we will write
$f\in\Omega_{\cF,D}$ as shorthand for $\gr f\in\Omega_{\cF,D}$.

\begin{Def}
  We say that $(\cS,\Omega)$ has \emph{sharp derivatives} if for every
  $\cF\in\N$ there are
  \begin{align}
    a_\cF&\in\N, &  b_\cF&\in\N[D,k]
  \end{align}
  such that the following holds. Given a definable $f:\R^n\to\R$ with
  $f\in\Omega_{\cF,D}$, we have for every $\alpha\in\Z_{\ge0}^n$
  \begin{equation}
    f^{(\alpha)}\in\Omega_{a_\cF,b_\cF(D,|\alpha|)}.
  \end{equation}
\end{Def}

Here $f^{(\alpha)}$ denotes the function with domain of definition
equal to the interior of the locus where $f$ is continuously
differentiable to order $|\alpha|$.

\begin{Rem}
  In every \so-minimal structure we have
  $f^{(\alpha)}\in\Omega_{a_{\cF,|\alpha|},b_{\cF,|\alpha|}(D)}$ with
  $b_{\cF,|\alpha|}\in\N[D]$. Sharp derivatives means that as we take
  derivatives of high order, the format remains fixed and the degree
  depends polynomially on the order. We do not know whether this holds
  for general \so-minimal structures.
\end{Rem}

\section{The restricted sub-Pfaffian structure $\R_\rPfaff$}
\label{sec:rPfaff}

In this section we let $\Omega$ denote the \so-minimal filtration on
$\R_\rPfaff$ introduced in \cite{bv:rpfaff}. The main result of
loc. cit. is that $(\R_\rPfaff,\Omega)$ is a \so-minimal structure
admitting sharp cell decomposition.

\begin{Rem}
  A small technical issue is that in \cite{bv:rpfaff} we considered
  only subsets of $[0,1]^n$, whereas in the o-minimal setting it is of
  course customary to work in $\R^n$. It is a routine matter to
  translate the results of \cite{bv:rpfaff} to this alternative
  context. Say $\Omega'$ denotes the FD-filtration introduced in
  \cite{bv:rpfaff}. Fix an algebraic diffeomorphism $\phi:\R\to(0,1)$,
  and by abuse of notation also write $\phi:\R^n\to(0,1)^n$ for
  $\phi^{\times n}$. Then one can define
  \begin{equation}
    \Omega_{\cF,D} := \{\phi^{-1}(X) : X\in\Omega_{\cF,D}'\},
  \end{equation}
  and deduce \so-minimality and sharp cell decomposition for
  $(\R_\rPfaff,\Omega)$ from the results of \cite{bv:rpfaff} for
  $\Omega'$.

  Another small issue is that the *-format and *-degree introduced in
  \cite{bv:rpfaff} does not exactly satisfy the axioms of
  \so-minimality as defined in~\secref{sec:so-minimality}, though this
  is a matter of a simple re-indexing. To obtain a \so-minimal
  structure one can consider $\Omega_{\cF,D}$ to be given by the
  collection of restricted sub-Pfaffian sets defined by first-order
  formulas of *-format $\cF$ and *-degree $D$, as defined in
  \cite[Definition~7]{bv:rpfaff}.
\end{Rem}

In the remainder of the section we will prove the following.

\begin{Thm}\label{thm:rPfaff-sharp-der}
  The structure $(\R_\rPfaff,\Omega)$ has sharp derivatives.
\end{Thm}

Let $U\subset\R^n$ and $f:U\to\R$ with $f\in\Omega_{\cF,D}$ and
$\Gamma:=\gr f$. By the definition from \cite{bv:rpfaff},
\begin{equation}
  \Gamma = \cup_i \pi_{n+1} X_i^\circ
\end{equation}
where i) each $X_i^\circ$ is a connected component of a semi-Pfaffian
$X_i\subset\R^N$ of degree $\poly_\cF(D)$ for some $N=N(\cF)$, and
ii) the number of $X_i$ is $\poly_\cF(D)$. Moreover according to
\cite[Lemma~18]{bv:rpfaff} we may assume that the projection
$\pi_{n+1}\rest{X_i}$ has finite fibers.

Fix one $X=X_i$ with $\pi_{n+1}(X_i^\circ)$ of full dimension in
$\Gamma$. The general case easily reduces to this at the end. By
\cite{gv:stratification} we may further assume that $X$ is effectively
smooth, i.e. is defined by a semi-Pfaffian system
\begin{equation}\label{eq:semipfaff-system}
  \{F_1=\ldots=F_{N-n}=0\}\cap\{G_1>0,\ldots,G_M>0\}
\end{equation}
with the differential $\d F_1\wedge\cdots\wedge\d F_{N-n}$
non-vanishing on $X$. The degrees of the $F_i,G_j$
in~\eqref{eq:semipfaff-system} are $\poly_\cF(D)$. Removing a
smaller-dimensional part, we may assume that the projection
$\pi_n\rest{X}$ is everywhere submersive.

Denote the coordinates on $\R^N$ by $(x,y)$ where
\begin{align}
  x&=(x_1,\ldots,x_n), & y&=(y_1,\ldots,y_{N-n}).
\end{align}
By the implicit function theorem and our setup above, around every
point in $X$ one can express $y$ as a smooth function of $x$, and
\begin{equation}
  F(x,y)=0 \implies \pd{F}{x}+\pd{F}{y}\cdot \pd{y}{x}=0\implies \pd{y}{x} = -\bigg(\pd{F}{y}\bigg)^{-1} \pd{F}{x},
\end{equation}
where $F=(F_1,\ldots,F_{N-n})$. Note that $\pd{F}{y}$ is invertible
everywhere on $X$ by our setup. Note $f(x)=y_1(x)$ on
$\pi_n(X^\circ)$. Since $F$ is a vector of Pfaffian functions, all the
derivatives in the right hand side are again Pfaffian, and using
$A^{-1}=\det^{-1}(A)\operatorname{adj}(A)$ we can write each
$\pd{y_i}{x_j}$ in the form $P_{ij}/Q$ where $P_{ij}$ is a Pfaffian
function and $Q=\det\pd{F}{y}$. Using this, one can rewrite
$f^{(\alpha)}(x)=y_1^{(\alpha)}(x)$ as a ratio of Pfaffian functions
$P_\alpha/Q^{2|\alpha|}$ with
$\deg P_\alpha=\poly_\cF(D)\cdot|\alpha|$, the asymptotic constants
depending on the Pfaffian chain used to define $f$ (which are part of
the format $\cF$). This ratio is not formally Pfaffian, but adding a
variable $z$ and an equation $Q^{2|\alpha|}z=P_\alpha$ to the
equations of $X$ gives a set $Z\subset\R^{N+1}$ with a connected
component $Z^\circ$ lying over $X^\circ$, such that the projection of
$Z$ to $(x,z)$ is the graph of $f^{(\alpha)}$ over $\pi_{n}(X^\circ)$.

Recall that the union of $\poly_\cF(D)$ sets $X^\circ$ as above define
a dense subset of $\Gamma$. We have thus seen how to define a dense
subset of $\gr f^{(\alpha)}$ with format $O_\cF(1)$ and degree
$\poly_\cF(D,|\alpha|)$. By \so-minimality, the closure of this dense
subset, $\Gamma_\alpha$, has similarly bounded format and degree.

Finally, the open set $D_\alpha\subset U$ equal to the interior of the
locus where $\Gamma_\alpha$ is the graph of a continuously
differentiable function is also in
$\Omega_{O_\cF(1),\poly_\cF(D,|\alpha|)}$ by \so-minimality. Setting
\begin{equation}
  \Gamma_\alpha' = \Gamma_\alpha\cap \bigcap_{|\beta|<|\alpha|} D_\beta
\end{equation}
defines the graph of $y^{(\alpha)}$ with the correct domain of
definition, and the format and degree bounds follow by \so-minimality.

\section{Norms on $C^r$-functions}

For
\begin{equation}
  P=\sum_{|\alpha|\le r} a_\alpha t^\alpha\in\R[t_1,\ldots,t_m]
\end{equation}
we denote by $\cM P=\sum_{|\alpha|\le r} |a_\alpha| t^\alpha$ the
\emph{majorant}. We set $\norm{P}:=\cM P(1,\ldots,1)$.

For a $C^r$-smooth function $f:U\to\R$ on a domain $U\subset\R^m$ we
and $x_0\in U$ we denote by
\begin{equation}
  j_{x_0}^r f=\sum_{|\alpha|\le r} \frac{f^{(\alpha)}(x_0)}{\alpha!} t^\alpha
\end{equation}
the $r$-jet of $f$ at $x_0$. We define two norms on $f$ as follows,
\begin{align}
  \norm{f}_r &:= \max_{|\alpha|\le r} \norm{D^\valpha f}, & \norm{f}_{T,r} &:= \sup_{x\in U} \norm{j^r_xf}.
\end{align}
As in~\secref{sec:intro-Cr} we extend this to $F:U\to\R^n$ by
coordinate-wise maximum.

For our purposes these two norms are essentially equivalent. Indeed,
on the one hand we have
\begin{equation}\label{eq:Tr-less-T}
  \norm{f}_{T,r} \le e^m \norm{f}_r.
\end{equation}
On the other hand the following lemma is immediate.
\begin{Lem}\label{lem:T-less-Tr}
  Suppose $f:I^n\to \R$ with $\norm{f}_{T,r}\le 1$. Let
  $\phi:I^n\to I^n$ be a diagonal affine map with $\Im\phi$ a cube of
  side-length $1/r$. Then $\norm{f\circ\phi}_r\le 1$.
\end{Lem}
As a consequence of Lemma~\ref{lem:T-less-Tr}, given functions of unit
$(T,r)$-norm on $I^n$ we can always rescale to obtain bounded $r$-norms
using $r^n$ charts.

We usually state our results with $\norm{f}_r$, but in some cases
$\norm{f}_{T,r}$ is more technically convenient, mainly because of the
following submultiplicativity and subcompositionality properties..

\begin{Lem}\label{lem:submul-subcomp}
  The following estimates for products and compositions hold:
  \begin{enumerate}
  \item Let $f,g:U\to\R$ be $C^r$-smooth. Then
    $\norm{fg}_{T,r}\le\norm{f}_{T,r}\cdot\norm{g}_{T,r}$.
  \item Let $F:U\to\R^n$ and $g:V\to\R$ be $C^r$-smooth with
    $\Im f\subset V$. Suppose $\norm{F_i}_{T,r}\le 1$ for
    $i=1,\ldots,n$. Then $\norm{g\circ F}_{T,r}\le \norm{g}_{T,r}$.
  \end{enumerate}
\end{Lem}
\begin{proof}
  Part i follows from
  \begin{equation}
    [\cM j_x^r(fg)](1,\ldots,1) \le [\cM j_x^rf](1,\ldots,1)\cdot [\cM j_x^rg](1,\ldots,1)
  \end{equation}
  which holds since $j_x^r(fg)$ is just $j_x^r(f)j_x^r(g)$
  truncated to degree $r$. Part ii follows in a similar fashion, this
  time noting that $j^r_x(g\circ F)$ is just
  $j_{F(x)}^r(g)(j_x^rF_1,\ldots,j_x^rF_n)$ truncated to degree
  $r$.
\end{proof}

\section{Proof of the algebraic lemma}

We will prove the algebraic lemma in the following equivalent form
which is more suitable for an inductive argument. Below, we think of
maps $F:\Lambda\times I^n\to I^k$ as definable families of maps
$\{F_\lambda:I^n\to I^k\}_{\lambda\in\Lambda}$, where
$F_\lambda:=F(\lambda,\cdot)$.

\begin{Lem}\label{lem:F-lem}
  Let $r\in\N$ and $\e>0$. Let $F:\Lambda\times I^n\to I^k$ with
  $F_j\in\Omega_{\cF,D}$ for $j=1,\ldots,k$. Then there exists a
  collection $\{\phi_\eta:\Lambda\times I^n\to I^n\}$ such that for
  every $\lambda\in\Lambda$ we have i)
  $\norm{F_\lambda\circ \phi_{\eta,\lambda}}_r\le 1$, ii)
  $\cup_\eta \Im(F_\lambda \circ\phi_{\eta,\lambda})\sube\Im F_\lambda$,
  and iii)
  \begin{equation}
    \#\{\phi_\eta\} \le \poly_\cF(D,r,k,|\log\e|), \qquad \forall\eta: \phi_\eta\in\Omega_{O_\cF(1),\poly_\cF(D,r)}.
  \end{equation}
\end{Lem}

Lemma~\ref{lem:F-lem} implies the following family version of
Lemma~\ref{lem:alg-lem}.

\begin{Lem}\label{lem:alg-lem-family}
  Let $r\in\N$ and $\e>0$. Let $X\subset\Lambda\times I^n$ with
  $\mu:=\max_\lambda\dim X_\lambda$ and $X\in\Omega_{\cF,D}$. Assume
  $X$ has no empty fibers. Then there exists a collection
  $\{\phi_\eta:\Lambda\times I^\mu\to X\}$ such that for every
  $\lambda\in\Lambda$ we have i) $\norm{\phi_{\eta,\lambda}}_r\le 1$,
  ii) $\cup_\eta \Im\phi_{\eta,\lambda}\sube X_\lambda$, and iii)
  \begin{equation}
    \#\{\phi_\eta\} \le \poly_\cF(D,r,|\log\e|), \qquad \forall\eta: \phi_\eta\in\Omega_{O_\cF(1),\poly_\cF(D,r)}.
  \end{equation}
\end{Lem}

\begin{proof}
  First perform a cell decomposition of $\Lambda\times I^n$ compatible
  with $X$, to cover $X_\lambda$ by $\poly_\cF(D)$ images
  $\Im F_{\theta,\lambda}$ for $F_\theta:\Lambda\times I^\mu\to I^n$.
  The non-empty fibers are required to guarantee we can always do
  this. Then apply Lemma~\ref{lem:F-lem} to each of these maps. The
  collection of all resulting $F_\theta\circ\phi_\eta$ establishes the
  conclusion the lemma.
\end{proof}

\begin{Rem}
  Suppose $\Lambda=\Lambda_1\cup\cdots\cup\Lambda_N$ is a definable
  subdivision of $\Lambda$ with $N=\poly_\cF(D,r,k,|\log\e|)$ and
  $\Lambda_i\in\Omega_{O_\cF(1),\poly_\cF(D,r,|\log\e|)}$. Suppose we
  prove Lemma~\ref{lem:F-lem} for $F$ restricted to each $\Lambda_i$
  separately, say giving collections
  \begin{equation}
    \{\phi_{i,j}:\Lambda_i\times I^n\to I^n\}\qquad i=1,\ldots,N,\quad j=1,\ldots,M
  \end{equation}
  allowing repetitions to make these collections have the same size
  $M$. Then the collection $\{\phi_j:=\cup_i \phi_{i,j}\}_j$ proves
  Lemma~\ref{lem:F-lem} for $\Lambda$ (the degree and format bounds
  follow from \so-minimality). A similar remark applies for
  Lemma~\ref{lem:alg-lem-family}. In the proof below we will often use
  this subdivision argument without explicit mention.
\end{Rem}

To make the notation more suggestive, we sometimes denote the
coordinates on $\R^n$ by $(x,y_1,\ldots,y_{n-1})$. The proof of
Lemma~\ref{lem:F-lem} will occupy the remainder of this section. We
proceed by induction on $n$, treating the base case $n=1$ in the
following subsection. We record a simple lemma that is useful in many
stages of our argument.

\begin{Lem}
  Let $F:X\to Y$ be $1$-Lipschitz, and suppose
  $\{\phi_\eta:D_i\to X\}$ satisfies
  $\cup_\e\Im\phi_\eta\sube X$. Then
  $\cup_\eta\Im (F\circ\phi_\eta)\sube \Im F$.
\end{Lem}

In particular this implies that when every $F_\lambda$ is
$1$-Lipschitz we can replace the condition
$\cup_\eta \Im(F_\lambda \circ\phi_{\eta,\lambda})\sube\Im F_\lambda$
in Lemma~\ref{lem:F-lem} by
$\cup_\eta \Im\phi_{\eta,\lambda}\sube I^n$. We will often use this
remark after performing a pullback to satisfy the $1$-Lipschitz
condition.

\subsection{The case $n=1$}

The main difficulty in proving Lemma~\ref{lem:F-lem} is to get
polynomial growth with respect to $r$. For a fixed $r$ the classical
proof of Yomdin-Gromov gives the same statement, even with a true
cover in place of the $\e$-cover. We record below the $C^1$-version
that we will need in the sequel.

\begin{Lem}\label{lem:full-C1}
  Let $F:\Lambda\times I\to I^k$ with $F_i\in\Omega_{\cF,D}$. Then
  there exists a collection $\{\phi_\eta:\Lambda\times I\to I\}$ such
  that for every $\lambda\in\Lambda$ we have i)
  $\norm{F_\lambda\circ\phi_{\eta,\lambda}}_1\le 1$, ii)
  $\cup_\eta \phi_{\eta,\lambda}(I)=I$, and iii)
  \begin{equation}
    \#\{\phi_\eta\} \le \poly_\cF(D,k), \qquad \forall\eta: \phi_\eta\in\Omega_{O_\cF(1),\poly_\cF(D)}.
  \end{equation}
\end{Lem}
\begin{proof}
  Assume without loss of generality that $f(x)=x$ is among the
  $F_i$. Denote $(\cdot)'=\pd{}x(\cdot)$. Perform a cell decomposition
  of $\Lambda\times I$ compatible with the sets of zeros of all the
  functions $|F_i'|-|F_j'|$ for $i,j=1,\ldots,k$ as well as with the
  sets of points where $F_i'$ is undefined. We have $\poly_\cF(D,k)$
  cells, each in $\Omega_{O_\cF(1),\poly_\cF(D)}$.

  It will suffice to handle each cell separately. For cells of the
  form $C\odot\{a(\lambda)\}$ one can cover their image by a constant
  map, so consider a cell $C\odot(a(\lambda),b(\lambda))$. Since each
  $|F_i'|-|F_j'|$ is either identically vanishing or identically
  non-vanishing on the cell, there is one $F_i$, without loss of
  generality $F_1$, such that
  \begin{equation}\label{eq:F1-dominates}
    |F_1'|\ge|F_j'| \quad \forall j=2,\ldots,k
  \end{equation}
  uniformly over the cell. In particular $|F_1|\ge1$. Set
  \begin{equation}
    F_1(\lambda,I)=(A(\lambda),B(\lambda))
  \end{equation}
  and define $\tilde\phi:C\odot(A(\lambda),B(\lambda))\to I$ by
  $\tilde\phi(\lambda,s)=(\lambda,F_1^{-1}(s))$. By~\eqref{eq:F1-dominates}
  we have
  \begin{equation}
    \begin{aligned}
      |(F_j\circ\tilde\phi(\lambda,s))'| &= |F_{j,\lambda}'(\tilde\phi_\lambda(s))\tilde\phi_\lambda(s)'|\\
      &=|F_{j,\lambda}'(\tilde\phi_\lambda(s))/F_{1,\lambda}'(\tilde\phi_\lambda(s))|\le 1
    \end{aligned}
  \end{equation}
  so $\norm{F_\lambda\circ\tilde\phi_\lambda}_1\le1$. Finally, let
  $\phi:C\times I\to I$ be the pullback of $\tilde\phi$ by a linear
  rescaling map $C\times I\to C\odot(A(\lambda),B(\lambda))$. Since
  $(A(\lambda),B(\lambda))\subset I$ this only decreases derivatives,
  so the collection of maps $\phi$ thus obtained satisfies the
  conditions of the lemma.
\end{proof}

We first apply Lemma~\ref{lem:full-C1} to $F$. Pulling back $F$ by
each of the $\phi_\eta$ thus obtained, we may assume without loss of
generality that $\norm{F_\lambda}_1\le 1$ for every
$\lambda\in\Lambda$. In particular, each $F_\lambda$ is 1-Lipschitz
(with respect to the $\ell_\infty$-norm).

Perform a cell decomposition of $\Lambda\times I$ compatible with the
sets of zeros of all the functions $F_i^{(j)}$ for $i=1,\ldots,k$ and
$j=0,\ldots,r+1$, as well as the sets of points where these functions
are undefined. We have $\poly_\cF(D,k,r)$ cells, each in
$\Omega_{O_\cF(1),\poly_\cF(D,r)}$. It will suffice to handle each
cell separately. For cells of the form $C\odot\{a(\lambda)\}$ there is
nothing to prove, so we consider cells
$C\odot(a(\lambda),b(\lambda))$. Pulling back by the affine map
$C\odot I\to C\odot(a(\lambda),b(\lambda))$ only decreases
derivatives, so without loss of generality it now suffices to prove
Lemma~\ref{lem:F-lem} assuming each $F_\lambda$ is $1$-Lipschitz and
has constant-signed derivatives up to order $r+1$. The result now
follows from the following lemma.

\begin{Lem}\label{lem:discrete-Cr}
  Let $f:I\to I$ such that $f^{(j)}$ has constant sign for
  $j=0,\ldots,r+1$. Then for every $M>1$ and $j=0,\ldots,r$ we have
  \begin{equation}
    |f^{(j)}(x)| < M^j\text{ whenever }\dist(x,\partial I) > j/M.
  \end{equation}
\end{Lem}
\begin{proof}
  We proceed by induction, the case $j=0$ being vacuous. Suppose the
  claim is proved for $f^{(j)}$. Assume $f^{(j+1)}$ is
  weakly-increasing (or weakly-decreasing, which is analogous) and
  suppose toward contradiction that $f^{(j+1)}(x)\ge M^{j+1}$ for some
  $x$ with $\dist(x,\partial I) > (j+1)/M$ (the case
  $f^{(j+1)}(x)\le -M^{j+1}$ being analogous). Then
  $f^{(j+1)}>M^{j+1}$ throughout the interval $[x,x+1/M]$. Thus
  $f^{(j)}(x+1/M)-f^{(j)}(x)\ge M^j$. This contradicts the inductive
  hypothesis, since both $x$ and $x+1/M$ have distance at least $j/M$
  to $\partial I$, and $f^{(j)}$ is constant-signed and bounded in
  absolute value by $M^j$ at both points.
\end{proof}

It follows from Lemma~\ref{lem:discrete-Cr} that if $\phi:I\to I$ is
an affine translation with the length of $\phi(I)$ smaller than
$\dist(\phi(I),\partial I)/r$ then
$\norm{F_\lambda\circ\phi}_r\le 1$ for every $\lambda$. It is an
elementary exercise that with $\poly(r,|\log\e|)$ such maps we can cover
$I_\e:=(\e,1-\e)$. Finally, since $F_\lambda$ is $1$-Lipschitz for
every $\lambda\in\Lambda$ we have $F_\lambda(I_\e)\sube F_\lambda(I)$
so this covering satisfies the conditions of Lemma~\ref{lem:F-lem}.

We record a corollary of the proof above for later use, where we cover
the domain $I$ by linear maps but skip the $1$-Lipschitz preparation
step, so we only get an $\e$-cover of the domain $I$ but not
necessarily of the image under $F$.

\begin{Cor}\label{cor:F1-affine}
  Let $r\in\N$ and $\e>0$. Let $F:\Lambda\times I\to I^k$ with
  $F_j\in\Omega_{\cF,D}$ for $j=1,\ldots,k$. Then there exists a
  collection $\{\phi_\eta:\Lambda\times I\to I\}$ such that for every
  $\lambda\in\Lambda$ we have i)
  $\norm{F_\lambda\circ \phi_{\eta,\lambda}}_r\le 1$, ii)
  $\cup_\eta \phi_{\eta,\lambda}(I)\sube I$, iii)
  $\phi_{\eta,\lambda}$ is affine, and iv)
  \begin{equation}
    \#\{\phi_\eta\} \le \poly_\cF(D,r,k,|\log\e|), \qquad \forall\eta: \phi_\eta\in\Omega_{O_\cF(1),\poly_\cF(D,r)}.
  \end{equation}
\end{Cor}

\subsection{Reduction to bounded $y$-derivatives}

We now continue with the case of general $n$, assuming that the case
$n-1$ is already established.

Apply the inductive hypothesis to
$F:(\Lambda\times I_x)\times I_y^{n-1}\to I^k$ with $\e/2$ in place of
$\e$, viewing $x$ as an additional parameter. Let $\{\phi_\eta\}$ be
the resulting collection. It will suffice to establish the conclusion
of Lemma~\ref{lem:F-lem} with $\e/2$ in place of $\e$ and each
$F\circ(\lambda,x,\phi_\eta)$ in place of $F$. In other words, without
loss of generality it suffices to prove Lemma~\ref{lem:F-lem} assuming
that $\norm{F_\lambda(x,\cdot)}_r\le 1$ for every $\lambda\in\Lambda$
and every $x\in I$. Below we keep working with $\e$ rather than $\e/2$
to simplify notations, and we make similar reductions later in the
proof.

\subsection{Reduction to $1$-Lipschitz}
\label{sec:Lipschitz-reduction}

Recall the following lemma from \cite{me:yg-revisited}.

\begin{Lem}[\protect{\cite{me:yg-revisited}}] \label{lem:uniforly-bdd-Fx}
  Let $F:I_x\times I^{n-1}_y\to I$ be definable and suppose that for
  each $i=1,\ldots,n-1$ the derivative $f_{y_i}(x,y)$ is uniformly
  bounded over all $y$ where it is defined, for almost every
  $x\in I_x$. Then $f_x(x,y)$ is also uniformly bounded for all $y$
  where it is defined, for almost every $x\in I_x$.
\end{Lem}
\begin{proof}
  This is essentially \cite[Lemma~16]{me:yg-revisited}, but we sketch
  the argument to illustrate the connection with the material
  above. Suppose toward contradiction that $f_x(x,y)$ is unbounded in
  $y$ for every $x\in J\subset I$. Without loss of generality assume
  $J=I$. Then by definable choice one can choose a function
  $\gamma:I\to I^{n-1}$ such that each $f_{y_i}(x,\gamma(x))$ is
  defined and $f_x(x,\gamma(x))>M$ for every $x\in I$ (or
  $f_x(x,\gamma(x))<-M$, which is analogous). Moreover the format and
  degree of $\gamma$ are bounded independently of $M$. By
  Corollary~\ref{cor:F1-affine} applied to $\gamma$ and
  $f(x,\gamma(x))$ with $r=1$ and say $\e=0.1$ we find a subinterval
  $I'\subset I$ where both $\norm{\gamma'(x)}$ and $|f(x,\gamma(x))'|$
  are bounded from above by a constant independent of $M$. One can
  take the longest of the intervals $\phi_\eta(I)$ for example, which
  has length bounded from below uniformly in $M$. This is now a
  contradiction for $M\gg1$ because
  \begin{equation}
    M<f_x(x,\gamma(x)) = f(x,\gamma(x))'-\sum_{i=1}^{n-1} f_{y_i}(x,\gamma(x)) \gamma_i'(x)
  \end{equation}
  and the right-hand side is uniformly bounded.
\end{proof}

For $i=1,\ldots,k$ we define $S_i\subset\Lambda\times I^n$ by
\begin{equation}
  S_i:=\big\{(\lambda,x,y) : |(F_i)_x(\lambda,x,y)| \ge \frac12 \sup_{y'\in I^{n-1}}|(F_i)_x(\lambda,x,y')| \big\}
\end{equation}
where the supremum is taken over the points where
$(F_i)_x(\lambda,x,y')$ is defined. By
Lemma~\ref{lem:uniforly-bdd-Fx}, the supremum is finite, for each
$\lambda\in\Lambda$, for almost every $x$. For $x$ where the supremum
is infinite we consider that the condition is vacuous, i.e. every
$(\lambda,x,y)$ is included in $S_i$ in this case. Clearly
$S_i\in\Omega_{O_\cF(1),\poly_\cF(D)}$.

By definable choice we may choose subsets $\Gamma_i\subset S_i$ such
that $\Gamma_i$ contains exactly one $(\lambda,x,y)$ for every
$(\lambda,x)$. In particular, sharp cylindrical decomposition shows
that $\Gamma_i\in\Omega_{O_\cF(1),\poly_\cF(D)}$. By definition
$\Gamma_i$ is a graph of an $(n-1)$-tuple of functions
\begin{equation}
  \gamma_{i,1},\ldots,\gamma_{i,n-1}:\Lambda\times I_x\to I
\end{equation}
and $\gamma_{i,j}\in\Omega_{O_\cF(1),\poly_\cF(D)}$ as well.

We apply Lemma~\ref{lem:full-C1} to the tuple including the functions
$\gamma_i$ and $x$, as well as $F\circ(\lambda,x,\gamma_i)$ for every
$i=1,\ldots,k$. For every $\phi_\eta$ thus obtained let
\begin{equation}
  \Phi_\eta(\lambda,t,y)=(\lambda,\phi_\eta(\lambda,t),y).
\end{equation}
Denote $F_\eta:=F\circ\Phi_\eta$. It will suffice to establish the
conclusion of Lemma~\ref{lem:F-lem} for each $F_\eta$ in place of
$F$. Moreover we obviously still have
$\norm{F_{\eta,\lambda}(t,\cdot)}_r\le 1$ for every
$\lambda\in\Lambda$ and every $t\in I$. Denote
$\Gamma_{i,\eta}:=\Phi_\eta^{-1}(\Gamma_i)$. Then $\Gamma_{i,\eta}$ is
the graph of an $(n-1)$-tuple of functions
$\gamma_{i,j,\eta}:=\gamma_{i,j}\circ(\lambda,\phi_\eta)$ with
$\norm{\gamma_{i,j,\eta,\lambda}}_1\le 1$ for every
$\lambda\in\Lambda$. Note that for every
$(\lambda,t,y)\in\Gamma_{i,\eta}$ we have
\begin{multline}\label{eq:Gamma-i-cond}
  |(F_{i,\eta})_t(\lambda,t,y)| = |(F_i)_x\circ\Phi_\eta(\lambda,t,y)\cdot (\phi_\eta)_t(\lambda,t)| \ge\\
  \frac12 \sup_{y'\in I^{n-1}} |(F_i)_x(\lambda,\phi_\lambda(t),y')\cdot (\phi_\eta)_t(\lambda,t)| =\\
  \frac12 \sup_{y'\in I^{n-1}} |(F_{i,\eta})_t(\lambda,t,y')|.
\end{multline}
In other words, $\Gamma_{i,\eta}$ satisfies the same definition in
the $(\lambda,t,y)$ coordinates as $\Gamma_i$ in the $(\lambda,x,y)$
coordinates. Clearly the sets and functions defined above have format
$O_\cF(1)$ and degree $\poly_\cF(D)$.

We claim that $|(F_{i,\eta})_t|=O_n(1)$ whenever it is
defined. According to~\eqref{eq:Gamma-i-cond} it is enough to
check this on the curves $\Gamma_{i,\eta}$. We compute the
derivative of $F_{i,\eta}$ along this curve,
\begin{multline}
  1 \ge |\big(F_{i,\eta}(\lambda,t,\gamma_{i,1,\eta,\lambda},\ldots,\gamma_{i,n-1,\eta,\lambda})\big)_t| = \\
  |(F_{i,\eta})_t+\sum_{j=1}^{n-1} (F_{i,\eta})_{y_j} (\gamma_{i,j,\eta,\lambda})_t| \ge
  |(F_{i,\eta})_t|- n-1,
\end{multline}
and rearranging we see that $|(F_{i,\eta})_t|=O_n(1)$ as claimed.

According to the following lemma, each $F_\eta(\lambda,\cdot)$
is $O_n(1)$-Lipschitz for each $\lambda\in\Lambda$ after a subdivision of
$I_x$ into intervals.

\begin{Lem}
  Let $f:I_x\times I_y^{n-1}\to I$ be definable with
  $\norm{f(x,\cdot)}_1\le 1$ everywhere and $|f_x|\le 1$ whenever
  $f_x$ is defined. Then the locus of discontinuity of $f$ is
  contained in finitely many hyperplanes $x=x_1,\ldots,x_N$, and $f$
  is $O_n(1)$-Lipschitz on each $(x_i,x_{i+1})\times I^{n-1}_y$ (where we
  take $x_1=0$ and $x_N=1$).
\end{Lem}
\begin{proof}
  Since $f$ is $O_n(1)$-Lipschitz in the $y$-direction, it is easy to check
  that if $f$ is discontinuous at $(x_0,y_0)$ it is also discontinuous
  at every point of $\{x_0\}\times U_\delta(y_0)$ for some
  $\delta\ll1$. Since the locus of discontinuity has empty interior,
  the first claim follows.

  Let $V\subset(x_i,x_{i+1})\times I^{n-1}_y$ denote the locus where
  $f_x$ is undefined. Consider two points in
  $(x_i,x_{i+1})\times I^{n-1}_y$ and the straight line $\gamma$
  connecting them. If $\gamma\cap V$ is finite then $f\rest\gamma$ is
  piecewise $O_n(1)$-Lipschitz and continuous, so it is $1$-Lipschitz. In
  general, we deform $\gamma$ into a curve $\gamma+v$ with
  $v\in\R^n$. For the same reason as above, we can choose $\norm{v}$
  arbitrarily small so that $(\gamma+v)\cap V$ is finite, and $f$ is
  $O_n(1)$-Lipschitz on $\gamma+v$. The claim follows by continuity of
  $f$.
\end{proof}

Perform a cell decomposition of $\Lambda\times I_x$ compatible with
the projections of the loci of discontinuity of $F_{i,\eta}$ for each
$i=1,\ldots,k$ to $\Lambda\times I_x$, giving cells of the form either
$C=C_\lambda\odot\{a(\lambda)\}$ or
$C=C_\lambda\odot(a(\lambda),b(\lambda))$ where in the latter case
$F_{\eta,\lambda}$ is $O_n(1)$-Lipschitz in
$(a(\lambda),b(\lambda))\times I^{n-1}$ for every
$\lambda\in\Lambda$. In the former case we can handle
$F_\eta\rest{C\times I^{n-1}}$ by the inductive hypothesis. In the
latter case, rescaling $(a(\lambda),b(\lambda))$ back to $(0,1)$ only
improves the Lipschitz constant in
$F_{\eta,\lambda}\rest{C\times I^{n-1}}$. Applying a further linear
subdivision in the $x,y$ coordinates we may further reduce the
Lipschitz constant to $1$ to simplify our notations.

Returning to the original notation, we conclude that it will suffice
to establish the conclusion of Lemma~\ref{lem:F-lem} assuming that
$\norm{F_\lambda(x,\cdot)}_r\le 1$ for every $\lambda\in\Lambda$ and
every $x\in I$, and $F_\lambda$ is $1$-Lipschitz for every $\lambda$.

\subsection{Controlling higher derivatives}
\label{sec:higher-der}

We start off similarly to~\secref{sec:Lipschitz-reduction}. For
$i=1,\ldots,k$ and $\alpha\in\Z_{\ge0}^n$ with $|\alpha|\le r$ we
define $S_{i,\alpha}\subset\Lambda\times I^n$ by
\begin{equation}
  S_{i,\alpha}:=\big\{(\lambda,x,y) : |F_i^{(\alpha)}(\lambda,x,y)| \ge
  \frac12 \sup_{y'\in I^{n-1}}|F_i^{(\alpha)}(\lambda,x,y')| \big\}
\end{equation}
where the supremum is restricted to those points where
$F_i^{(\alpha)}(\lambda,x,y')$ is defined. Repeatedly applying
Lemma~\ref{lem:uniforly-bdd-Fx} and using the fact that for
$\alpha_1=0$ all the $F^{(\alpha)}$ are defined and bounded
everywhere, we again conclude that the supremum is finite, for each
$\lambda\in\Lambda$, for almost every $x$. For $x$ where the supremum
is infinite we consider that the condition is vacuous, i.e. every
$(\lambda,x,y)$ is included in $S_{i,\alpha}$ in this case. By the
sharp derivatives condition one sees that
$S_{i,\alpha}\in\Omega_{O_\cF(1),\poly_\cF(D,r)}$.

By definable choice we may choose subsets
$\Gamma_{i,\alpha}\subset S_{i,\alpha}$ such that $\Gamma_{i,\alpha}$
contains exactly one $(\lambda,x,y)$ for every $(\lambda,x)$. In
particular, sharp cylindrical decomposition shows that
$\Gamma_{i,\alpha}\in\Omega_{O_\cF(1),\poly_\cF(D,r)}$. By definition
$\Gamma_{i,\alpha}$ is a graph of an $(n-1)$-tuple of functions
$\gamma_{i,\alpha}$ given by
\begin{equation}
  \gamma_{i,\alpha,1},\ldots,\gamma_{i,\alpha,n-1}:\Lambda\times I_x\to I
\end{equation}
and $\gamma_{i,\alpha,j}\in\Omega_{O_\cF(1),\poly_\cF(D,r)}$ as well.

We apply Corollary~\ref{cor:F1-affine} to the tuple including the
functions $\gamma_{i,\alpha}$ as well as
$F^{(\beta)}\circ(\lambda,x,\gamma_{i,\alpha})$ for every
$i=1,\ldots,k$, every $\alpha\in\Z_{\ge0}^n$ satisfying
$|\alpha|\le r$, and every $\beta\in\Z_{\ge0}^n$ satisfying
$|\beta|\le r$ and $\beta_1=0$. Note that $F^{(\beta)}$ is indeed
bounded by $1$ according to our previous steps, so the corollary
applies. For every $\phi_\eta$ thus obtained let
\begin{equation}
  \Phi_\eta(\lambda,t,y)=(\lambda,\phi_\eta(\lambda,t),y).
\end{equation}
Denote $F_\eta:=F\circ\Phi_\eta$. It will suffice to establish the
conclusion of Lemma~\ref{lem:F-lem} for each $F_\eta$ in place of $F$:
indeed, $\cup_\eta\phi_{\eta,\lambda}(I)\sube I$ for every
$\lambda\in\Lambda$, and since $F_\lambda$ is $1$-Lipschitz for every
$\lambda$ it follows that
$\cup_\eta F_\lambda\circ\Phi_{\eta,\lambda}(I^n)\sube F_\lambda(I^n)$. Moreover
we obviously still have $\norm{F_{\eta,\lambda}(t,\cdot)}_r\le 1$ for
every $\lambda\in\Lambda$ and every $t\in I$.

Denote
$\Gamma_{i,\alpha,\eta}:=\Phi_\eta^{-1}(\Gamma_{i,\alpha})$. Then
$\Gamma_{i,\alpha,\eta}$ is the graph of an $(n-1)$-tuple
$\gamma_{i,\alpha,\eta}:=\gamma_{i,\alpha}\circ(\lambda,\phi_\eta)$ of
functions
$\gamma_{i,\alpha,j,\eta}:=\gamma_{i,\alpha,j}\circ(\lambda,\phi_\eta)$
with $\norm{\gamma_{i,\alpha,\eta,\lambda}}_r\le 1$ for every
$\lambda\in\Lambda$. Denote $\pd{}t(\cdot)=(\cdot)'$ and recall that
\begin{equation}\label{eq:phi-eta-affine}
  \phi_\eta''=0
\end{equation}
since these maps are affine in $t$. Then for every
$(\lambda,t,y)\in\Gamma_{i,\alpha,\eta}$ we have
\begin{multline}\label{eq:Gamma-i-alpha-cond}
  |F_{i,\eta}^{(\alpha)}(\lambda,t,y)| = |F_i^{(\alpha)}\circ\Phi_\eta(\lambda,t,y)\cdot (\phi_\eta')^{\alpha_1}(\lambda,t)| \ge\\
  \frac12 \sup_{y'\in I^{n-1}} |F_i^{(\alpha)}(\lambda,\phi_\lambda(t),y')\cdot (\phi_\eta')^{\alpha_1}(\lambda,t)| =\\
  \frac12 \sup_{y'\in I^{n-1}} |F_{i,\eta}^{(\alpha)}(\lambda,t,y')|.
\end{multline}
In other words, $\Gamma_{i,\alpha,\eta}$ satisfies the same definition
in the $(\lambda,t,y)$ coordinates as $\Gamma_{i,\alpha}$ in the
$(\lambda,x,y)$ coordinates. Note that~\eqref{eq:phi-eta-affine} was
crucial at this point: otherwise we would get numerous additional
terms involving mixed derivatives of $F_i$ and $\phi_\eta$. As before,
the sets and functions defined above have format $O_\cF(1)$ and degree
$\poly_\cF(D,r)$.

\begin{Lem}\label{lem:main-inductive-computation}
  Let $\lambda\in\Lambda$. Let $\Gamma$ be one of
  $(\Gamma_{i,\alpha,\eta})_\lambda$ and $\gamma:I\to\Gamma$ be the
  tuple $\gamma_{i,\alpha,\eta,\lambda}$. Let $G=F_{l,\eta,\lambda}$
  for some $l=1,\ldots,k$. Then for each $\beta\in\Z_{\ge0}^n$ with
  $|\beta|\le r$ we have
  \begin{equation}\label{eq:G-norm-est}
    \norm{G^{(\beta)}\circ(t,\gamma)}_{T,r-\beta_1} \le E(\beta_1), \qquad E(\beta_1):=e\cdot(r+(n-1)e)^{\beta_1}.
  \end{equation}
\end{Lem}
\begin{proof}
  We will work by induction on $\beta_1$. For $\beta_1=0$ the
  estimate~\eqref{eq:G-norm-est} follows from the fact that
  $F^{(\beta)}\circ(\lambda,x,\gamma_{i,\alpha})$ was included in the
  tuple of functions to which Corollary~\ref{cor:F1-affine} was
  applied. Indeed,
  $F^{(\beta)}\circ\Phi_\eta=(F\circ\Phi_\eta)^{(\beta)}$ and it
  follows that
  \begin{equation}
    \begin{aligned}
      G^{(\beta)}\circ(t,\gamma) &= (F_l\circ\Phi_\eta)^{(\beta)}\circ(\lambda,t,\gamma_{i,\alpha,\eta,\lambda}) \\
      &= F_l^{(\beta)}\circ\Phi_\eta\circ(\lambda,t,\gamma_{i,\alpha,\eta,\lambda}) \\
      &= F_l^{(\beta)}\circ(\lambda,\phi_\eta(\lambda,t),\gamma_{i,\alpha}\circ\phi_\eta(\lambda,t)) \\
      &= F_l^{(\beta)}\circ(\lambda,x,\gamma_{i,\alpha})\circ\phi_\eta(\lambda,t).
    \end{aligned}
  \end{equation}
  so the $(T,r)$-norm of the left hand side is bounded according to
  Corollary~\ref{cor:F1-affine} and~\eqref{eq:Tr-less-T}.

  Suppose now that~\eqref{eq:G-norm-est} is proved for all $\beta'$
  with $\beta'_1<\beta_1$. Compute
  \begin{equation}\label{eq:G-on-curve}
    (G^{(\beta-e_1)}\circ(t,\gamma))' = G^{(\beta)}\circ(t,\gamma)+\sum_{j=2}^n (G^{(\beta-e_1+e_j)}\circ(t,\gamma))\cdot\gamma_j',
  \end{equation}
  where $e_1,\ldots,e_n\in\Z_{\ge0}^n$ and $e_{i,j}=\delta_{i,j}$. By
  the inductive hypothesis
  \begin{equation}
    \norm{G^{(\beta-e_1)}\circ(t,\gamma)}_{T,r-\beta_1+1}\le E(\beta_1-1).
  \end{equation}
  Thus the left-hand side of~\eqref{eq:G-on-curve} has
  $(T,r-\beta_1)$-norm bounded by $rE(\beta_1-1)$. Similarly
  \begin{equation}
    \norm{G^{(\beta-e_1+e_j)}\circ(t,\gamma))}_{T,r-\beta_1}\le E(\beta_1-1)
  \end{equation}
  by induction and $\norm{\gamma_j'}_{T,r-\beta_1}\le e$ since
  $\norm{\gamma_j}_r\le 1$. Rearranging and using
  Lemma~\ref{lem:submul-subcomp} we conclude that
  \begin{equation}
    \norm{G^{(\beta)}\circ(t,\gamma)}_{T,r-\beta_1} \le r E(\beta_1-1)+(n-1)e E(\beta_1-1)
  \end{equation}
  as claimed.
\end{proof}

Finally we conclude that
$|F_{i,\eta,\lambda}^{(\alpha)}|=O_n(r^{|\alpha|})$ whenever it
is defined, for any $\lambda\in\Lambda$, any $i=1,\ldots,k$ and any
$|\alpha|\le r$.  Indeed, according to~\eqref{eq:Gamma-i-alpha-cond}
it is enough to check this on the curve $\Gamma_{i,\alpha,\eta}$, and
there it holds by Lemma~\ref{lem:main-inductive-computation} taking
$G=F_{i,\eta,\lambda}^{(\alpha)}$ since the $(T,r)$-norm bounds the
maximum norm. A further linear subdivision into cubes of length $1/r$
as in Lemma~\ref{lem:T-less-Tr} then gives
$\norm{F_{i,\eta,\lambda}}_r\le 1$, whenever the derivatives are
defined.

Returning again to the original notation, we conclude that it will
suffice to establish the conclusion of Lemma~\ref{lem:F-lem} assuming
that $F_\lambda$ is $1$-Lipschitz and $\norm{F_\lambda}_r\le 1$ for
every $\lambda\in\Lambda$.

\subsection{Final clean up}

By now we have satisfied the conclusions of Lemma~\ref{lem:F-lem},
except that some derivatives of $F$ may be undefined at some points.
Let $V_{i,\alpha}$ denote the locus where $F_i^{(\alpha)}$ is
undefined, for $i=1,\ldots,k$ and $|\alpha|\le r$.  Perform a cell
decomposition of $\Lambda\times I^n$ compatible with every
$V_{i,\alpha}$ giving $\poly_\cF(D,r,k)$ cells, each in
$\Omega_{O_\cF(1),\poly_\cF(D,r)}$. This induces a cell decomposition
$\{C_\nu\}$ of $\Lambda\times I^{n-1}$.

By Lemma~\ref{lem:alg-lem-family} we get a collection
$\{\phi_{\nu,\eta}:\Lambda\times I^{n-1}\to C_\nu\}$ such that
\begin{equation}
  \cup_\eta(\phi_{\nu,\eta,\lambda}(I^{n-1}))\sube C_{\nu,\lambda}
\end{equation}
and $\norm{\phi_{\nu,\eta,\lambda}}_r\le 1$. Moreover,
\begin{equation}
  \#\{\phi_{\nu,\eta}\} \le \poly_\cF(D,r,|\log\e|), \qquad \forall\nu,\eta: \phi_{\nu,\eta}\in\Omega_{O_\cF(1),\poly_\cF(D,r)}.
\end{equation}
Since $F_\lambda$ is 1-Lipschitz for every $\lambda$ it will be enough
to prove the claim for each pullback $F\circ(\phi_{\nu,\eta},x_n)$,
and we may further restrict to the case where $C_{\nu,\lambda}$ has
full dimension in $I^{n-1}$. Then
$\norm{F_\lambda\circ(\phi_{\nu,\eta,\lambda},x_n)}_r$ is bounded for
every $\lambda$ by Lemma~\ref{lem:submul-subcomp}, and by linear
subdivision as in Lemma~\ref{lem:T-less-Tr} one can return to unit
$r$-norms (and after further subdivision to $1$-Lipschitz).

In other words we may replace $F$ by each
$F\circ(\phi_{\nu,\eta},x_n)$ and assume without loss of generality
that $V:=\cup_{i,\alpha} V_{i,\alpha}$, i.e. the locus of
non-smoothness of $F$, is already given by graphs of functions
$G_1,\ldots,G_q\in\Omega_{O_\cF(1),\poly_\cF(D,r)}$ with
$q=\poly_\cF(D,k,r)$, where $G_i:\Lambda\times I^{n-1}\to I$. We also
may assume $G_1=0$ and $G_q=0$ for simplicity.

Now apply the inductive hypothesis, i.e. Lemma~\ref{lem:F-lem} in
dimension $n-1$, to the tuple $x_1,\ldots,x_{n-1}$ and to
$G_1,\ldots,G_q$ on $\Lambda\times I^{n-1}$. Making the same reduction
as above, we may now assume without loss of generality that
$\norm{G_{1,\lambda}}_r,\ldots,\norm{G_{q,\lambda}}_r\le1$ for every
$\lambda$. We cover
\begin{equation}
  (\Lambda\times I^n) \setminus V = \bigcup_{i=1}^q (\Lambda\times I^{n-1})\odot(G_i,G_{i+1})
\end{equation}
by images of affine maps
\begin{equation}
  (\lambda,x_1,\ldots,x_{n-1},t) \to (\lambda,x_1,\ldots,x_{n-1},tG_{i+1}+(1-t)G_i)
\end{equation}
with similarly bounded $r$-norms, and finally by another pullback step as
above reduce to the case $V=\emptyset$, finishing the proof.

\section{Point counting}

In this section we give the proof of
Theorem~\ref{thm:main-blocks}. Since the argument is fairly standard
by now we focus mostly on the novel parts. The key proposition is the
following.

\begin{Prop}\label{prop:interpolation-step}
  There are appropriate choices of $r,d=\poly_m(g,\log H)$ and
  $\log\e=-\poly_m(g,\log H)$ such that the following holds. Let
  $\phi:I^m\to X$ such that
  $\phi(I^m)\sube X\subset I^{m+1}$ and $\norm{\phi}_r\le1$. Then
  there exists a polynomial $P\in\R[x_1,\ldots,x_{m+1}]\setminus\{0\}$
  of degree $d$ such that $X(g,H)\subset\{P=0\}$.
\end{Prop}
\begin{proof}
  There are two approaches to proving such a statement. The first, due
  to Bombieri-Pila \cite{bp}, is based on interpolation
  determinants. The second, due to Wilkie \cite{wilkie:notes} is based
  on the Siegel lemma. Both of these are normally stated with
  $\phi(I^m)=X$. We briefly show that for both methods the weaker
  assumption $\phi(I^m)\sube X$ is sufficient. After cutting into
  $2^m$ pieces we may assume that the domain of $\phi$ is $J^m$ where
  $J:=(0,1/2)$ instead of $I^m$.

  We start with the interpolation determinant method, which directly
  applies to the case $g=1$, i.e. to $X(\Q,H)$. Recall that an
  \emph{interpolation determinant} is
  \begin{equation}
    \Delta^d(p_1,\ldots,p_\mu) = \det (p_i^\alpha) \quad \text{for }\ {\substack{i=1,\ldots,\mu\\ \alpha\in\Z_{\ge0}^{m+1}, |\alpha|\le d}}
  \end{equation}
  where $p_i\in\R^{m+1}$ and $\mu$ is the dimension of the space of
  polynomials of degree at most $d$ in $m+1$ variables. To prove
  $X(\Q,H)\subset\{P=0\}$ it is enough to show that $\Delta^d(\cdot)$
  vanishes for any $\mu$-tuple of points in $X(\Q,H)$. This is done as
  follows. First, for any such tuple $p$ one estimates the height of
  $\Delta^d(p)$, concluding that either $\Delta^d(p)=0$ or
  $|\Delta^d(p)|\ge H^{-d\mu}$. On the other hand an analytic estimate
  shows that for an appropriate choice of $r,d$ as above we have
  $|\Delta^d(p)|< \tfrac12 H^{-d\mu}$. These contradicting estimates
  force $\Delta^d(p)=0$ on $\phi(J^m)$ and finish the proof.

  In our context, we would like to extend this to the case
  $p_1,\ldots,p_\mu\in U_\e(\phi(J^m))$. Let
  $q_1,\ldots,q_\mu\in\phi(J^m)$ with $\dist(p_i,q_i)\le\e$. Then
  $\Delta^d(q)< \tfrac12 H^{-d\mu}$ as above, and if show
  $|\Delta^d(q)-\Delta^d(p)|< \tfrac12 H^{-d\mu}$ we can finish as
  above. One easily estimates
  $\norm{\partial\Delta^d(p)/\partial p}\le d\mu!$ in
  $I^{\mu\times(n+1)}$, so choosing
  $\e\sim\tfrac12 H^{-d\mu}/[d(\mu+2)!]$ suffices.

  We now consider Wilkie's approach. Using Siegel's lemma one
  constructs a polynomial $P\in\Z[x_1,\ldots,x_{m+1}]$ of degree $d$
  and coefficients bounded by some $N$ such that
  $P(\phi_1,\ldots,\phi_{m+1})$ has many small Taylor
  coefficients. Liouville's inequality gives for any $x\in X(g,H)$
  that either $P(x)=0$ or
  \begin{equation}
    |P(x)| \ge (d^{m+1}N H^{d(m+1)})^g.
  \end{equation}
  Denote the right-hand side by $R$. Wilkie shows that for an
  appropriate choice of $r,d$ as above (and also
  $\log N=\poly_m(g,\log H)$) we have $|P(x)|<R/2$ whenever
  $x\in\phi(J^m)$, thus forcing $P$ to vanish on $x$. In fact Wilkie
  considered $X(\Q,H)$ and the one-dimensional case, but see
  \cite[Proposition~16]{habegger:approximation} or
  \cite[Proposition~28]{me:qfol-geometry} for a treatment of the
  general case.

  If we take $x\in U_\e(\phi(I^m))$ instead, and $x_0\in\phi(I^m)$
  with $\dist(x,x_0)\le\e$, then as above it will suffice to show that
  $|P(x_0)-P(x)|<R/2$. The bounds on $d,N$ easily imply
  $\log \norm{\partial P/\partial x}<\poly_m(g,\log H)$ in $[0,1]^n$, so choosing
  an appropriate $\log\e=-\poly_m(g,\log H)$ suffices to finish the proof.
\end{proof}

We proceed to the proof of Theorem~\ref{thm:main-blocks}, which is
very similar to \cite{pila:blocks} and \cite{bjst:effective-pw} modulo
the sharper Proposition~\ref{prop:interpolation-step}. We proceed by
induction on dimension $m:=\dim X$, the zero-dimensional case being
trivial by \so-minimality. Suppose the claim is proved for $X$ of
dimension smaller than $m$.

Up to inverting and negating some coordinates (which does not affect
height) one can cut $X$ into pieces contained in $[0,1]^n$, so assume
this without loss of generality. Apply Lemma~\ref{lem:alg-lem} to $X$
to obtain a collection $\{\phi_\eta\}$ of size $\poly_\cF(D,g,\log H)$
and $\cup_\eta\phi_\eta(I^m)\sube X$ where $m=\dim X$. It will be
enough to consider each $\phi_\eta(I^m)$ separately, so fix one
$\phi=\phi_\eta$ and suppose $\phi(I^m)\sube X$. Using
Proposition~\ref{prop:interpolation-step} we can find for each
$J\subset\{1,\ldots,n\}$ of size $m+1$ a polynomial $P_J$ of degree
$d=\poly_n(g,\log H)$ in the coordinates $(x_i:i\in J)$ vanishing on
$X(g,H)$. The zero loci of these polynomials cut out an algebraic
variety $V$ of dimension at most $m$ in $\R^n$. Stratify $V$
(e.g. using \so-minimality of $\R^\alg$) and let $\{S_i\}$ denote the
top dimensional strata and $S'$ the union of the rest. Note
$S_i,S'\in\Omega^\alg_{O_n(1),\poly_n(d)}$ by \so-minimality. The
points in $S'\cap X$ are handled by induction on $m$. Now stratify
$X\cap S_i$ and denote by $\{B_{ij}\}$ the top dimensional strata and
by $B'$ the union of the rest (over all $S_i$). Note $\#\{B_{ij}\}$ is
again $\poly_\cF(D,g,\log H)$ by \so-minimality. Finally $B'\cap X$ is
similarly handled by induction on $m$, while $B_{ij}$ are by
definition basic blocks with semialgebraic closures $S_i$, which
finishes the proof.

\bibliographystyle{plain}
\bibliography{nrefs}

\end{document}